\documentclass[a4paper,12pt,twoside]{amsart}

\usepackage[cm]{fullpage}

\usepackage{mymacros2,oitp3}
\usepackage[
pagebackref]{hyperref}

\includecomment{comment}

\usepackage{comment} 

\title[Homological mirror symmetry for K3 surface]
{On homological mirror symmetry
for the complement of a smooth ample divisor
in a K3 surface}
\author[Y.~Lekili]{Yank\i\ Lekili}
\address{
Department of Mathematics,
Imperial College London,
South Kensington,
London,
SW7 2AZ}
\email{y.lekili@imperial.ac.uk}
\author[K.~Ueda]{Kazushi Ueda}
\address{
Graduate School of Mathematical Sciences,
The University of Tokyo,
3-8-1 Komaba,
Meguro-ku,
Tokyo,
153-8914,
Japan.}
\email{kazushi@ms.u-tokyo.ac.jp}
\date{}

\begin{document}

\begin{abstract}
We introduce a conjecture on homological mirror symmetry
relating the symplectic topology of the complement of a smooth ample divisor
in a K3 surface to algebraic geometry of type III degenerations,
and prove it when the degree of the divisor is either $2$ or $4$.
\end{abstract}

\maketitle

Let $X$ be a K3 surface and $D$ be a smooth ample divisor.
Since the affine variety
$
U \coloneqq X \setminus D
$
is simply connected
and has the trivial canonical bundle,
the wrapped Fukaya category $\cW(U)$
(and its full subcategory
$\cF(U)$ consisting of compact Lagrangians)
comes with a unique $\bZ$-grading.
The symplectomorphism type of the pair $(X,D)$ is determined
by the degree $d \coloneqq D \cdot D$
and the index $k$ of primitivity
(i.e., the maximal $k$ such that $[D]/k \in H_2(X;\bZ)$).
The degree $d$ can be any positive even integer,
and the index $k$
can be any positive integer
satisfying $2 k^2 | d$.

A \emph{Type III K3 surface}
is a surface $X$
obtained from a collection
$\lb X_i, D_i \rb_{i=1}^n$
of smooth rational surfaces $X_i$
with anticanonical cycles $D_i \in |-K_{X_i}|$
by identifying irreducible components of $D_i$
in such a way that
the dual intersection complex is a triangulation of $S^2$.
Each irreducible double curve $C \subset D_i \cap D_j$
must satisfy the \emph{triple point formula}
\begin{align}
  \lb C \rb^2_{X_i} + \lb C \rb^2_{X_j} =
  \begin{cases}
    0 & \text{$C$ is a nodal curve}, \\
    -2 & \text{$C$ is smooth}.
  \end{cases}
\end{align}

We call
the number of triangles
in the dual intersection complex
and the index of primitivity
of the monodromy logarithm
of the degeneration
(which is also determined by the dual intersection complex
\cite[Theorem 6.5]{MR813580})
as the \emph{degree} and the \emph{index}
of the Type III K3 surface.
Any Type III K3 surfaces with the same degree and index
are related by a sequence of Type I and Type II modifications
shown in Figures \ref{fg:type1} and \ref{fg:type2}
\cite[Theorem 0.6.3]{MR813580}.
\begin{figure}[h]
\begin{minipage}{.45 \linewidth}
  \begin{center}
    \begin{tikzpicture}[scale=.6]
      \draw (0,1) -- (1,0) -- (0,-1);
      \draw (1,0) -- node[above] {$a$} node[below] {$b$} (3,0);
      \draw (4,1) -- (3,0) -- (4,-1);
      \draw [->] (5,0) -- (6,0);
      \draw (7,1) -- (8,0) -- (7,-1);
      \draw (8,0) -- node[above] {$a-1$} node[below] {$b+1$} (11,0);
      \draw (12,1) -- (11,0) -- (12,-1);
    \end{tikzpicture}
  \end{center}
  \caption{Type I modification}
\label{fg:type1}
\end{minipage}
\begin{minipage}{.45 \linewidth}
  \begin{center}
    \begin{tikzpicture}[scale=.6]
      \draw (0,1) -- (1,0) -- (0,-1);
      \draw (1,0) -- node[above] {$-1$} node[below] {$-1$} (3,0);
      \draw (4,1) -- (3,0) -- (4,-1);
      \draw [->] (5,0) -- (6,0);
      \draw (7,2) -- (8,1) -- (9,2);
      \draw (8,1) -- node[left] {$-1$} node[right] {$-1$} (8,-1);
      \draw (7,-2) -- (8,-1) -- (9,-2);
    \end{tikzpicture}
  \end{center}
  \caption{Type II modification}
\label{fg:type2}
\end{minipage}
\end{figure}

Here, the integer on an edge of the intersection complex
denotes the self-intersection number of the corresponding
irreducible double curve.
These modifications are given by the Atiyah flop on the total space of the smoothing,
and hence do not change the derived category $\coh X$ of coherent sheaves
by \cite[Theorem 1.1]{MR2593258}.  

We fix the complex structure of a Type III K3 surface
by the following two conditions:
\begin{enumerate}
  \item For any $i \in \{ 1, \ldots, n \}$,
  the complex structure of the log Calabi--Yau surface $X_i \setminus D_i$
  is the one appearing in \cite{2005.05010}.
  This means that a suitable corner blow-up of $(X_i, D_i)$
  is obtained from a toric surface
  by blowing up (possibly infinitely near) points on the toric boundary
  whose toric coordinates are $-1$.
  \item Irreducible components of the anticanonical cycles are glued torically.
  This means that
  when two irreducible curves on the boundaries are glued together,
  the centers of the blow-ups are identified.
\end{enumerate}

\begin{conjecture} \label{cj:main}
Let $U$ be the complement of a smooth ample divisor
of degree $d$ and index $k$ in a K3 surface.
Let further $X$ be a Type III K3 surface
of degree $d$ and index $k$.
Then there exist equivalences
\begin{align} \label{eq:wfuk=coh}
  \cW(U) \simeq \coh X
\end{align}
and
\begin{align} \label{eq:fuk=perf}
  \cF(U) \simeq \perf X.
\end{align}
\end{conjecture}

\begin{theorem} \label{th:main}
\pref{cj:main} holds in degree $2$ and $4$.
\end{theorem}

\begin{proof}
The Milnor fiber
\begin{align}
  U \coloneqq \lc (x,y,z) \in \bC^2 \relmid x^2 + y^6 + z^6 = 1 \rc
\end{align}
of the Brieskorn--Pham singularity
obtained as the suspension
of the Sebastiani--Thom sum of two copies of the $A_5$-singularity
is the complement of a smooth ample divisor
of degree $2$ in a K3 surface.

Let
$
\bP \coloneqq \ld \lb \bA^4 \setminus \bszero \rb \middle/ \Gamma \rd
$
be the quotient stack of
the complement of the origin $\bszero$
in $\bA^4 = \Spec \bC[x,y,z,w]$
by the diagonal action of
\begin{align}
  \Gamma &\coloneqq \lc (\alpha, \beta, \gamma, \delta) \in (\Gm)^4 \relmid
  \alpha^2 = \beta^6 = \gamma^6 = \alpha \beta \gamma \delta \rc,
\end{align}
and $Z$ be the hypersurface of $\bP$
defined by
$
x^2+y^6+z^6+x y z w.
$
It is shown
in \cite[Theorem 1.7.(ii)]{1806.04345}
that
\begin{align}
  \cW(U) \simeq \coh Z
\end{align}
and
\begin{align}
  \cF(U) \simeq \perf Z.
\end{align}

Let $\bPt$ be a crepant resolution of the coarse moduli scheme of $\bP$,
and $X$ be the hypersurface of $\bPt$
defined by
$
 x^2 + y^6 + z^6 + x y z w.
$
One has
\begin{align}
  \coh Z \simeq \coh X
\end{align}
and
\begin{align}
  \perf Z \simeq \perf X
\end{align}
by (the proof of) \cite[Theorem 2.1]{MR2593258}.

Note that $Z$ has a unique singular point
$[0:0:0:1]$
contained in the open substack of $\bP$
defined by $w \ne 0$
whose coarse moduli scheme $V$ is isomorphic to
\(
 \bA^3 \left/ \la \frac{1}{2}(1,0,1), \frac{1}{6}(0,1,5) \ra \right..
\)
One can choose $\bPt$ in such a way
that the inverse image of $V$
is the toric variety
whose fan is the cone over the triangulation
given in \pref{fg:K3_degree2_fan1}.
Then the scheme $X$
consists of the strict transform $X_1$ of $Z$
and two more irreducible components $X_2$ and $X_3$,
which are the compact toric divisors
in the crepant resolution of $V$.

\begin{figure}
\begin{minipage}{.5 \linewidth}
\begin{center}
  \begin{tikzpicture}
    \foreach \y in {0,...,6}
    {
  \filldraw (0,\y) circle [radius=1mm];
  }
  \foreach \y in {0,...,3}
  {
  \filldraw (1,\y) circle [radius=1mm];
  }
  \filldraw (2,0) circle [radius=1mm];
  \draw (0,0) -- (2,0) -- (0,6) -- cycle;
  \draw (1,1) -- (0,0);
  \draw (1,1) -- (1,0);
  \draw (1,1) -- (2,0);
  \draw (1,1) -- (1,2);
  \draw (1,1) -- (0,3);
  \draw (1,1) -- (0,2);
  \draw (1,1) -- (0,1);
  \draw (1,2) -- (2,0);
  \draw (1,2) -- (1,3);
  \draw (1,2) -- (0,6);
  \draw (1,2) -- (0,5);
  \draw (1,2) -- (0,4);
  \draw (1,2) -- (0,3);
  \end{tikzpicture}
\end{center}
\caption{A crepant resolution of $\bA^3\left/ \la \frac{1}{2}(1,0,1), \frac{1}{6}(0,1,5) \ra \right.$}
\label{fg:K3_degree2_fan1}
\end{minipage}
\begin{minipage}{.45 \linewidth}
  \begin{center}
    \begin{tikzpicture}
      \draw (1,1) -- (0,0) node[below left] {$E_3$};
      \draw (1,1) -- (1,0) node[below] {$E_1-E_3$};
      \draw (1,1) -- (2,0) node[below right] {$H-E_1$};
      \draw (1,1) -- (1,2) node[above right] {$H-E_2-E_4$};
      \draw (1,1) -- (0,3) node[above] {$E_4$};
      \draw (1,1) -- (0,2) node[above left] {$E_2-E_4$};
      \draw (1,1) -- (0,1) node[below left] {$H-E_1-E_2-E_3$};
    \end{tikzpicture}
    \end{center}
\caption{A weak toric del Pezzo surface of degree 5}
\label{fg:K3_degree2_fan2}
\end{minipage}
\end{figure}

An explicit description in terms of toric coordinates shows that
the family defined by
$
x^2+y^6+z^6+t w^6+xyzw=0
$
inside $\bPt$
is a semistable degeneration of K3 surfaces
whose central fiber is $X$.

Each of $X_2$ and $X_3$
is the toric surface
obtained by blowing up $\bP^2$ four times
as designated by the fan
given in \pref{fg:K3_degree2_fan2}.
The double curve $C = X_2 \cap X_3$
is described on each of $X_2$ and $X_3$
as the $(-1)$-curve
obtained as the strict transform
$H-E_2-E_4$
of the line passing through the pair of infinitely near points
which are the centers of the successive blow-ups.
It follows that
the intersection complex of $X$ must be
the one shown in \pref{fg:bigon1}.
\begin{figure}
  \begin{center}
    \begin{tikzpicture}[scale=1.5]
      \draw (0,0) .. controls (1,1) and (2,1) .. node[above] {$-4$} node[below] {$2$} (3,0);
      \draw (0,0) -- node[above] {$-1$} node[below] {$-1$} (3,0);
      \draw (0,0) .. controls (1,-1) and (2,-1) .. node[above] {$2$} node[below] {$-4$} (3,0);
    \end{tikzpicture}
  \end{center}
  \caption{The division of $S^2$ into three bigons}
  \label{fg:bigon1}
\end{figure}
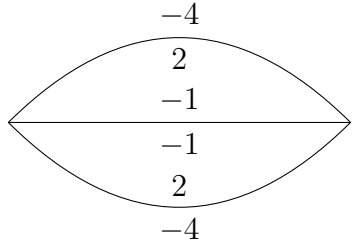
The dual intersection complex is a triangulation of $S^2$
by two triangles.

In order to determine the complex structure of $X_1$,
let $Y$ be the graph of the rational map
\begin{align}
  Z \dashrightarrow \bP^2, \quad
  [x:y:z:w] \mapsto [p:q:r]=[x^2:y^6:z^6],
\end{align}
i.e., the closure in $Z \times \bP^2$
of the graph of the regular map
from the complement of $x=y=z=0$ in $Z$ to $\bP^2$.
The projection $Y \to \bP^2$ is the stacky weighted blow-up
of $\bP^2$ of weights $(1,2)$, $(1,6)$, and $(1,6)$
at the intersections of the coordinate lines
and the line defined by $p+q+r=0$.
The minimal resolution $\Yt$ of the coarse moduli scheme of $Y$
is a rational surface with an anticanonical cycle
whose irreducible components have
self-intersection numbers $-1$, $-5$, and $-5$.
By blowing down the $(-1)$-component
of the anticanonical cycle,
one obtains the rational surface $X_1$ with an anticanonical cycle
consisting of two $(-4)$-curves.
This concludes the proof of \pref{th:main}
for degree $2$.

The story for the degree 4 case is completely parallel.
The complement of a smooth ample divisor of degree 4 in a K3 surface
is the Milnor fiber of the Brieskorn--Pham singularity
obtained as the Sebastiani--Thom sum
of three copies of the $A_3$-singularity.
Let $\bP$ be the quotient stack of
$\bA^4 \setminus \bszero$
by the diagonal action of
\begin{align}
  \Gamma &\coloneqq \lc (\alpha, \beta, \gamma, \delta) \in (\Gm)^4 \relmid
  \alpha^4 = \beta^4 = \gamma^4 = \alpha \beta \gamma \delta \rc,
\end{align}
and $Z$ be the hypersurface of
$
\bP
$
defined by
$
x^4+y^4+z^4+x y z w.
$
It is shown
in \cite[Theorem 1.7.(i)]{1806.04345}
that
\begin{align}
  \cW(U) \simeq \coh Z
\end{align}
and
\begin{align}
  \cF(U) \simeq \perf Z.
\end{align}
A stacky resolution $Y$ of $Z$ is obtained as the graph of the rational map
\begin{align}
  Z \dashrightarrow \bP^2, \ 
  [x:y:z:w] \mapsto [x^4:y^4:z^4].
\end{align}
The projection $Y \to \bP^2$ is the weighted blow-up
of $\bP^2$ of weight $(1,4)$
at the three intersection points of the coordinate lines
and a general line.
The minimal resolution of the coarse moduli scheme of $Z$
is a rational surface with an anticanonical cycle
consisting of three $(-3)$-curves.
The scheme $X$ consists of this surface
and three more irreducible components,
which are the compact toric divisors
in the crepant resolution of
$
\bA^3/\la \frac{1}{4}(1,3,0), \frac{1}{4}(0,1,3) \ra
$
whose fan is shown in \pref{fg:K3_degree4_fan1}.
\begin{figure}
\begin{minipage}{.5 \linewidth}
\begin{center}
\begin{tikzpicture}
\foreach \x in {0,1,2,3,4}
{
\filldraw (\x,0) circle [radius=1mm];
}
\foreach \x in {0,1,2,3}
{
\filldraw (\x,1) circle [radius=1mm];
}
\foreach \x in {0,1,2}
{
\filldraw (\x,2) circle [radius=1mm];
}
\foreach \x in {0,1}
{
\filldraw (\x,3) circle [radius=1mm];
}
\filldraw (0,4) circle [radius=1mm];
\draw (0,0) -- (4,0) -- (0,4) -- cycle;
\draw (0,2) -- (2,2);
\draw (0,1) -- (3,1);
\draw (1,0) -- (1,3);
\draw (2,0) -- (2,2);
\draw (1,1) -- (0,0);
\draw (1,1) -- (0,2);
\draw (1,1) -- (2,0);
\draw (2,1) -- (3,0);
\draw (2,1) -- (4,0);
\draw (2,1) -- (1,2);
\draw (1,2) -- (0,3);
\draw (1,2) -- (0,4);
\draw (1,2) -- (2,1);
\end{tikzpicture}
\end{center}
\caption{A crepant resolution of $\bA^3/\la \frac{1}{4}(1,3,0), \frac{1}{4}(0,1,3) \ra$}
\label{fg:K3_degree4_fan1}
\end{minipage}
\begin{minipage}{.45 \linewidth}
  \begin{center}
    \begin{tikzpicture}
      \draw (0,0) -- (1,0) node[right] {$H-E_1-E_2$};
      \draw (0,0) -- (0,1);
      \node [anchor=south] at (0.5,1) {$H-E_3-E_4$};
      \draw (0,0) -- (-1,1) node[above left] {$E_4$};
      \draw (0,0) -- (-1,0) node[left] {$E_3-E_4$};
      \draw (0,0) -- (-1,-1) node[below left] {$H-E_1-E_3$};
      \draw (0,0) -- (0,-1) node[below] {$E_1-E_2$};
      \draw (0,0) -- (1,-1) node[below right] {$E_2$};
    \end{tikzpicture}
    \end{center}
    \caption{A weak del Pezzo surface of degree 5}
    \label{fg:K3_degree4_fan2}
\end{minipage}
\end{figure}
Each of these irreducible components
is $\bP^2$ blown-up at four points
(including infinitely near points)
whose fan is shown in \pref{fg:K3_degree4_fan2}.
The double curves among these two components have
self intersection minus one on each irreducible components.
The intersection complex of $X$ is the tetrahedron
shown in \pref{fg:tetrahedron}.
\begin{figure}
\begin{center}
\begin{tikzpicture}[scale=3]
  \draw (0,0) -- node[left] {$-1$} node[right] {$-1$} (1,1);
  \draw (0,0) -- node[above] {$-1$} node[below] {$-1$} (-1,0);
  \draw (0,0) -- node[left] {$-1$} node[right] {$-1$} (0,-1);
  \draw (-1,0) -- node[above left] {$-3$} node[below right] {$1$} (1,1);
  \draw (0,-1) -- node[above left] {$1$} node[below right] {$-3$} (1,1);
  \draw (-1,0) -- node[below left] {$-3$} node[above right] {$1$} (0,-1);
\end{tikzpicture}
\end{center}
\caption{The tetrahedron}
\label{fg:tetrahedron}
\end{figure}
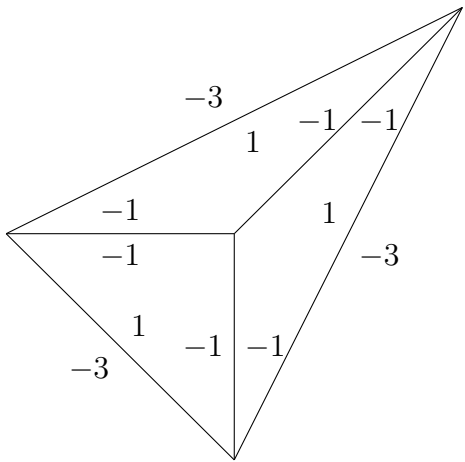
\end{proof}

\pref{th:main}
combined with
\cite{MR2819674}
proves 
\cite[Conjecture 1.5]{1806.04345}
for a smooth ample divisor of degree $2$ in a K3 surface.

\begin{corollary}
One has
\begin{align}
  \cF(D) \simeq \cW(U)/\cF(U)
\end{align}
when $U = X \setminus D$ is the complement
of a smooth ample divisor $D$
in a K3 surface $X$.
\end{corollary}

\begin{proof}
Let $S$ be the hypersurface of 
\(
  \Spec \bC[x,y,z] / \la \frac{1}{5}(1,1,3) \ra
\)
defined by
\begin{align}
  x^5 + y^5 + z^5 + x y z = 0.
\end{align}
The mirror of the genus two curve
is identified in \cite{MR2819674}
as the total transform $H$ of $S$
in a crepant resolution of
\(
  \Spec \bC[x,y,z] / \la \frac{1}{5}(1,1,3) \ra.
\)
The surface $H$ has three irreducible components $H_i$ for $i=1,2,3$.
The components
$
H_1 \cong \bP^2
$
and
$
H_2 \cong \bP(\cO_{\bP^1} \oplus \cO_{\bP^1}(-3))
$
are the exceptional divisors
of the resolution, and
the component $H_3$ is the strict transform of $S$.
In order to determine the complex structure of $H_3$,
let $T$ be the graph of the rational map
\begin{align}
  S \dashrightarrow \bP^2, \quad
  (x,y,z) \mapsto [p:q:r]=[x:y:z^2].
\end{align}
The projection $T \to \bP^2$ is the restriction
to an open subset
of a blow-up
of ten points on the conic $pq=r^2$
defined by $p^5+q^5=0$
followed by a weighted blow-up of weight $(1,2)$
of five points on the coordinate line $r=0$
defined by $p^5+q^5=0$.
It follows that $H_3$ is an open subscheme
of the blow-up of ten points on a conic
and five pairs of infinitely near points on a line
in $\bP^2$.

The compactification of the surface $H$,
whose intersection complex is shown in \pref{fg:Seidel_g2},
is related to the type III K3 surface $X$
of degree 2
appearing in the proof of \pref{th:main}
by deformation (i.e., moving the center of the blow-up to infinitely near points at $p+q=0$)
and a sequence of Type I modifications.
Since the deformation does not change the formal neighborhood of the singular locus
and the Type I modification does not change the derived category,
the stable derived categories of $X$ and $H$ are equivalent.
\end{proof}

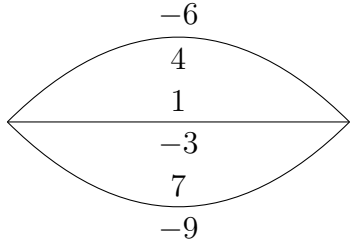
\begin{figure}
  \begin{center}
    \begin{tikzpicture}[scale=1.5]
      \draw (0,0) .. controls (1,1) and (2,1) .. node[above] {$-6$} node[below] {$4$} (3,0);
      \draw (0,0) -- node[above] {$1$} node[below] {$-3$} (3,0);
      \draw (0,0) .. controls (1,-1) and (2,-1) .. node[above] {$7$} node[below] {$-9$} (3,0);
    \end{tikzpicture}
  \end{center}
  \caption{Seidel's mirror for the genus 2 curve}
  \label{fg:Seidel_g2}
\end{figure}



\begin{remark}
Homological mirror symmetry for 2-dimensional pair of pants
(see e.g.~\cite{MR3364859,
MR2863919,
1604.00114,
MR3838112,
MR4120165}
and references therein)
implies an equivalence
\begin{align}
  \cW(U) \simeq \coh X
\end{align}
where
\begin{align}
  U \coloneqq \lc (x,y,z) \in (\bCx)^3 \relmid x+y+z+1/xyz=0 \rc
\end{align}
is the complement of a nef divisor of degree 64
and index 4
with 24 nodes
in a K3 surface
(the quartic mirror)
and $X$ is the toric boundary of $\bP^3$.
Note that $X$ can be turned into a Type III K3 surface
by blowing up at 24 points
consisting of four points on each of six double curves.
It is natural to
expect that
smoothing a node at the divisor at infinity
is mirror to
blowing up a point
on the double curve.
\end{remark}

\bibliographystyle{amsalpha}
\bibliography{bibs}

\end{document}